\documentclass[reqno,11pt]{amsart}

\usepackage{amsmath,amsfonts,amssymb,amsthm}
\usepackage [latin1]{inputenc}
\usepackage{latexsym,bm,graphicx}
\usepackage{mathrsfs}
\usepackage{color}

\newtheorem{thm}{Theorem}[section]
\newtheorem{prop}[thm]{Proposition}
\newtheorem{lem}[thm]{Lemma}
\newtheorem{cor}[thm]{Corollary}

\theoremstyle{definition}
\theoremstyle{remark}

\newtheorem{defn}[thm]{Definition}
\newtheorem{remark}[thm]{Remark}

\numberwithin{equation}{section}

\newcommand{\ls}{\leqslant}
\newcommand{\gs}{\geqslant}

\newcommand{\bR}{\mathbb{R}}

\newcommand{\cH}{\mathcal{H}}
\newcommand{\cL}{\mathcal{L}}

\def\XXint#1#2#3{{\setbox0=\hbox{$#1{#2#3}{\int}$ }
\vcenter{\hbox{$#2#3$ }}\kern-.6\wd0}}

\title{Non-existence of concave functions on certain metric spaces}
\author{Yin Jiang}
\address{School of Mathematical Sciences, Beihang University, Beijing P.R.C.}
\email[Yin Jiang]{jiangyin@buaa.edu.cn}

\begin{document}
\maketitle

\bibliographystyle{amsplain}

\begin{abstract}
In the paper \cite{yau1974convex}, Yau proved that: There is no non-trivial continuous concave function on a complete manifold with finite volume. We prove analogue theorems for several metric spaces, including Alexandrov spaces with curvature bounded below/above, $C^{\alpha}$-H\"older Riemannian manifolds.

\end{abstract}
\maketitle
\section{Introduction}
In the paper \cite{yau1974convex}, Yau proved that: There is no non-trivial continuous concave function on a complete manifold with finite volume. Yau's theorem generalized a theorem of Bishop and O'Neill [BO], which says that there is no non-trivial smooth function on a complete manifold with finite volume.
The "volume" of an n dimensional Riemannian manifold coincides with the n dimensional Hausdorff measure $\cH^n$. Hausdorff measure and concave functions are basic concepts of metric spaces. it's natural to ask whether analogue theorems holds on metric spaces.

Let $(X,|\cdot,\cdot|)$ be a metric space. A geodesic in $X$ is an isometric embedding of an interval into $X$. A metric space is called a geodesic space if any two points $p,q$ can be connected by some geodesic. Let $(X,|\cdot,\cdot|)$ be a geodesic space, $f:X\to \bR$ is called a concave function if for any geodesic $\gamma(t)$, $f\circ\gamma(t)$ is a concave function with respect to t.

There are plenty of concave functions on certain Alexandrov spaces. For non-compact n dimensional Alexandrov spaces with non-negative curvature, without boundary, the Busemann function $\lim\limits_{t\to \infty} d(x,\gamma(t))-t$ is a concave function. In the paper \cite{shiohama1993}, Shiohama asked whether a complete Alexandrov n-space admitting a non-trivial continuous concave function has unbounded Hausdorff n volume. We give an affirm answer to this question. We will prove that
\begin{thm}
Let $(X,|\cdot,\cdot|)$ be a complete n dimensional SCBBL, without boundary. If $X$ admits a non-trivial continuous concave function, then the n dimensional Hausdorff volume $\cH^n(X)=\infty$.
\end{thm}
 Where "SCBBL" means Alexandrov spaces with curvature bounded below locally: any point $x\in X$ has a neighborhood $U_x$ with curvature $\geqslant k_x$. $k_x$ may vary from point to point. They are more general than Alexandrov spaces with curvature $\gs k$ for some $k\in \bR$, and include $C^2$ Riemannian manifolds. See section 2 for precise definition. The condition "without boundary" is necessary in the above theorem, since for any complete n dimensional Alexandrov space with non-negative curvature, with boundary $N\neq \emptyset$, the distance function $d(N,\cdot)$ is concave. However if the space is compact, it has finite Hausdroff measure. So if the space $X$ has boundary, we should consider the double, and assume that the nature extension of $f$ on the double is concave.

We will also consider Alexandrov spaces with curvature bounded above. A complete simply connected space of nonpositive curvature is called a Hadamard space. If $(X,|\cdot,\cdot|)$ is a Hadamard space, then for every point $p\in X$, the function $x\to -|px|$ is a concave function defined on $X$. We adopt the notation "GCBA" of paper \cite{LyN2016} with a slight modification, $X$ is said to be GCBA, if $X$ is a locally compact, locally geodesically complete space with curvature bounded above. Compared with \cite{LyN2016}, we don't assume that $X$ is separable here. Note that a complete GCBA is separable and geodesically complete.
We will proved that:
\begin{thm}
 Let $(X,|\cdot,\cdot|)$ be a complete, finite dimensional (suppose the dimension is n) GCBA, if $X$ admits a non-trivial, continuous concave function, then $\cH^n(X)=\infty$.
\end{thm}
Locally geodesically complete is necessary. Consider a disc with boundary, $X=\{(x,y)\in \bR^2:x^2+y^2 \ls 1\}$, $-d(0,x)$ is a concave function on $X$, the distance function $dist_{\partial X}$ is also concave. $X$ has finite Hausdorff measure, $X$ is not locally geodesically complete.

We also consider H\"older manifolds, we will prove that
\begin{thm}
Let $(X,g)$ be a $C^{1,\alpha}$ differential manifold with a $C^{\alpha}$ Riemannian metric $g$, without boundary. If $X$ admits a non-trivial locally Lipschitz, concave function, then the n dimensional Hausdorff measure $\cH^n(X)=\infty$.
\end{thm}
Note that, before Yau's theorem, it's already known that, a concave function on a $C^2$ Riemannian manifold is locally Lipschitz  (see e.g. \cite{rockafellar1970}).

Yau proved his theorem by Liouville's theorem, i.e. geodesic flow preserves the volume element of sphere bundle. It's unknown whether Liouville's theorem holds for the spaces we mentioned in the above theorems, see \cite{AKPflow} for recent results. We will prove the above theorems by gradient flow, backward gradient flow and coarea inequality. All the above theorems are proved in the same way. The technique of gradient flow takes its roots in the so called Sharafutdinov's retraction, see \cite{PePe1993}. For Alexandrov spaces with curvature bounded below, it was first used by G.Perelman and A.Petrunin in \cite{PePe1993}. Bit later, for spaces with curvature above, it was used independently by J. Jost \cite{Jost1998} and U. Mayer \cite{Mayer1998}. Later, A. Lytchak unified and generalized these two approaches to a wide class of metric spaces, called "appropriate spaces" in \cite{lytchakopen}, which includes Alexandrov spaces, $C^{\alpha}$ H\"older Riemannian manifolds in the above theorems. For these spaces, differential of a locally Lipschitz, semiconcave function is well defined. For Alexandrov space with curvature bounded below, continuous concave function is locally Lipschitz. This also holds for GCBA, see Proposition \ref{prop:lip}.

We first study the backward gradient flow of the spaces mentioned in the theorems, and prove the above theorems for Lipschitz, concave functions. We mention that for locally Lipschitz, concave functions, only by backward gradient flow is not enough, see remark \ref{remark}. We use reparametrized backward gradient flow and coarea inequality to prove the theorems.

This paper is organized as follows:

In Section 2, we will provide some necessary materials for Alexandrov spaces with curvature bounded below locally, Alexandrov spaces with curvature bounded above, $C^{\alpha}$ H\"older Riemannian manifolds, semiconcave functions, gradient flow of these spaces.

In Section 3, we will prove that if the spaces are compact (without boundary), then any locally Lipschitz, concave function defined on the whole space is a constant. We will study backward gradient curves of these spaces and prove the theorems for Lipschitz, concave functions.

In section 4, we will introduce coarea inequalities for locally Lipschitz functions.

In section 5, we will prove the theorems for locally Lipschitz, concave functions.

\textbf{Acknowledgments} We are grateful to Prof. HuiChun Zhang for sharing his notes on Alexandrov spaces with curvature bounded below. His proof of the existence of backward gradient curves is listed in section 3. His proof is by constructing broken geodesics and take a limit, not by homology, and is suitable for our use. We also would like to thank Prof. V. Kapovitch for helpful discussion on backward gradient curves of Alexandrov spaces with curvature bounded below. He recommended \cite{LyN2016} and gave many helpful suggestions. The author was partially supported by NSFC 11901023.

\section{Preliminaries}

\subsection{Alexandrov spaces with curvature bounded below locally} $\ $

Let $(X,|\cdot \cdot|)$ be a metric space. A geodesic in $X$ is an isometric embedding of an interval into $X$. A metric space is called a geodesic space if any two points $p,q \in X$ can be connected by a geodesic.
Denote by $M_k^2$ the simply connected 2-dimensional space form of constant curvature $k$. Given three points $p,q,r$ in a metric space $(X,|\cdot,\cdot|)$, we can take a comparison triangle $\Delta \tilde{p}\tilde{q}\tilde{r}$ in $M^2_k$, such that
$$
d(\tilde{p},\tilde{q})=|pq|,d(\tilde{p},\tilde{r})=|pr|,d(\tilde{q},\tilde{r})=|qr|.
$$
If $k>0$, we add the assumption $|pq|+|pr|+|qr|<2\pi/\sqrt{k}$. The angle $ \widetilde{\angle}_k pqr: =\angle \tilde{p}\tilde{q}\tilde{r}$ is called a comparison angle.

\begin{defn}
An intrinsic metric space $(X,|\cdot,\cdot|)$ is called an Alexandrov space with curvature bounded below locally (SCBBL) if for any point $x\in X$, there exists a neighborhood $U_x$ and a number $k_x\in \bR$, such that, for any four different points $p,a,b,c$ in $U_x$, we have
\begin{equation}\label{ieq:comparison}
\widetilde{\angle}_{k_x} apb +\widetilde{\angle}_{k_x} bpc +\widetilde{\angle}_{k_x} cpa \ls 2\pi.
\end{equation}
\end{defn}
In the definition above, $k_x$ may vary from point to point, and it may be possible that $\inf_{x\in X} k_x=-\infty$. Compare this definition with that of \cite{burago1992ad}.
\begin{remark}
 For a compact or non-compact, n dimensional differential manifold $M$ with a $C^2$ Riemannian metric $g$. Let $|\cdot,\cdot|$ be the distance induced from $g$. Then $(M,d)$ is an Alexandrov space with curvature bounded below locally.
\end{remark}

The following property is known to experts.
\begin{prop}\label{prop:global}
Let $X$ be a complete SCBBL, without boundary, then for any $p\in X$ and $R>0$, there is a real number $k\in \bR$, such that for any $x\in B(p,R)$, there is a neighborhood $U_x$ such that the comparison inequality (\ref{ieq:comparison}) holds for $M_k^2$.
\end{prop}
\begin{proof}
Suppose the contrary. Then for any positive integer $i$, there exists $x_i\in B(p,R)$ such that, for any neighborhood $U_{x_i}$ for $X$, there exists $a_i,b_i,c_i,d_i\in U_{x_i}$, the inequality (\ref{ieq:comparison}) doesn't hold for $M_{-i}^2$. Suppose that a subsequence $x_{i_j}$ converge to $x\in \overline{B(p,R)}$. Then there is a neighborhood $U_x$ and number $k_x\in \bR$, such that for any different points $p,a,b,c\in U_x$, the inequality (\ref{ieq:comparison}) holds for $M_{k_x}^2$. For $i_j$ sufficiently large with $-i_j<k_x$, we can choose $U_{x_{i_j}}\subset U_x$, then we get a contradiction.
\end{proof}

By repeating the arguments of Chapter 6 of paper \cite{burago1992ad}, we can get the following property:
\begin{prop}
Let $p,q\in X$, then for sufficiently small neighborhoods $U_p,U_q$, the Hausdorff dimension of $U_p$ is equal to $U_q$, and is an integer or infinity. We call it the Hausdorff dimension of $X$.
\end{prop}

Suppose that $X$ is a SCBBL, whose Hausdorff dimension is $n<\infty$. Denote by $\cH^n$ the n-dimensional Hausdorff measure. By repeating the arguments of Chapter 6 of \cite{burago1992ad}, we get that $X$ is locally compact. For a complete, locally compact intrinsic metric space any two points can be connected by a geodesic. Hence if $X$ is a complete, n dimensional SCBBL, then it is a geodesic space.

\subsection{Alexandrov Spaces with curvature bounded above} $\ $
In this subsection, the reader is referred to \cite{burago2001course}, \cite{LyN2016}. Denote $R_k=\frac{1}{\sqrt{k}}\pi$ if $k>0$ and $R_k=\infty$ if $k\ls 0$.
\begin{defn}
A space of curvature $\ls k$ is a length space $X$ which can be covered by $\{U_i\}_{i\in I}$ so that every $U_i$ satisfies:

1. every two points with distance $<R_k$ can be connected by a geodesic in $U_i$.

2. For any $a,b,c\in U_i$ with $|ab|+|bc|+|ac|<2R_k$ and a point $d$ in any geodesic $ac$, $|db|\ls |\tilde{d}\tilde{b}|$. Where $\Delta \tilde{a}\tilde{b}\tilde{c}$ is a comparison triangle for $\Delta abc$ in $M_k^2$ and $\tilde{d}$ is the point in geodesic $\tilde{a}\tilde{c}$ such that $|ad|=|\tilde{a}\tilde{d}|$.
\end{defn}
\begin{defn}
$(X,d)$ is called a space of curvature bounded above if every point $x\in X$ has a neighborhood where this two conditions are satisfied for some $k_x$. $k_x$ depends on $x$, may vary from one point to another.
\end{defn}

We say that $(X,d)$ is space with non-positive curvature if it is a space with curvature $\ls 0$.
\begin{defn}
A complete simply connected space of nonpositive curvature is called a Hadamard space.
\end{defn}
\begin{prop}[see e.g. Corollary 9.2.14 of \cite{burago2001course}] $\ $

If $X$ is a Hadamard space, then for every $p\in X$ the function $x\to -|px|$ is concave.
\end{prop}

\begin{defn}[Ref. Definition 4.1 of \cite{LyN2016}] $\ $
Let $(X,d)$ be a space with curvature bounded above.
We call $X$ locally geodesically complete if any local geodesic $\gamma:[a,b]\to X$, for any $a<b$, extends as a local geodesic to a larger interval $[a-\epsilon, b+\epsilon]$. If any local geodesic in $X$ can be extended as a local geodesic to $\bR$ then $X$ is called geodesically complete.
\end{defn}
A complete metric space with an upper curvature bound is geodesically complete if it's locally geodesically complete.

We adopt the notation "GCBA" of \cite{LyN2016} with slight modification, we say that $X$ is a GCBA, if $X$ is a locally compact, locally geodesically complete space with curvature bounded above.
For the Hausdorff dimension of GCBA, we have the following property:
\begin{prop}[Theorem 1.1 of \cite{LyN2016}] $\ $
Let $X$ be a separable GCBA, then the topological dimension coincides with the Hausdorff dimension. It equals the supremum of dimensions of open subsets of $X$ homeomorphic to Euclidean balls.
\end{prop}
Let $X$ be a separable GCBA, if its Hausdorff dimension is $n<\infty$, we say that $X$ is an n dimensional GCBA.

\subsection{Tangent cone, gradient curve} $\ $

For this subsection, we refer the reader to \cite{burago1992ad}, \cite{burago2001course},\cite{lytchakdiff},\cite{lytchakopen}, \cite{lytchakholder} for details.

\textbf{In the rest of this subsection, we always assume that $(X,d)$ is one of the three types spaces below:}
Complete n-dimensional SCBBL; Complete GCBA ; $C^{1,\alpha}$ differential manifolds with a $C^{\alpha}$-Riemannian metric $g$, where $0<\alpha<1$, without boundary, the distance $d$ is induced by the metric $g$.

For $p\in X$, if $\gamma_1,\gamma_2$ are two unit speed geodesics which start at $p$, then the angle $\angle(\gamma_1,\gamma_2)$ is well defined:
$$
\angle(\gamma_1,\gamma_2)=\lim_{s,t\to 0} \tilde{\angle}(\gamma_1(s),p,\gamma_2(t)),
$$
where
\[
\tilde{\angle}(\gamma_1(s),p,\gamma_2(t))=arccos \frac{s^2+t^2-|\gamma_1(s)\gamma_2(t)|^2}{2st}.
\]
Unit speed geodesics with origin $p$ are considered to be equivalent if they form a zero angle with each other. The set of equivalent classes of unit speed geodesics with origin $p$ endowed with the distance $\angle$ is a metric space, denoted by $\Sigma'_p$. Its metric completion is called the space of directions at $p$ and is denoted by $\Sigma_p$.
The Euclidean cone over the space of directions $\Sigma_p$ is called geodesic cone at $p$, denoted by $C_p$.

\begin{prop}\label{prop:2geo}
For each $\epsilon>0$ there is some $\rho>0$, such that for each $x$ with $|px|<\rho$, $\angle_p(\gamma, \eta)<\epsilon$ holds for all geodesics $\gamma, \eta$ connecting $p$ and $x$.
\end{prop}

\begin{prop}
Let $p\in X$. The Gromov-Hausdorff limit of pointed spaces $(\lambda X,p)$ as $\lambda \to \infty$ exists. The limit is called the Gromov-Hausdorff tangent cone of $X$ at $p$, denoted as $T_p$ or $T_p X$. It equals to the geodesic cone $C_p$.
\end{prop}
For Alexandrov spaces, see e.g. \cite{burago1992ad} and Chapter 9 of \cite{burago2001course}. For $C^{\alpha}$ H\"older Riemannian manifolds, see e.g. \cite{lytchakdiff} and \cite{lytchakholder}.

So we just call $T_p$ or $C_p$ tangent cone. For $C^{1,\alpha}$ manifold with $C^{\alpha}$ Riemannian metric, $T_p$ is just $\bR^n$ with standard metric. For two tangent vectors $u,v$ at $p$, the "scalar product" is defined as:
\[
\langle u, v \rangle := \frac{1}{2}(|u|^2+|v|^2-|uv|^2)=|u||v|\cos \alpha,
\]
where $\alpha=\angle uov=\tilde{\angle}_0 uov$ in $T_p$.

If $\gamma(t)$ is unit speed geodesic corresponding to $\xi\in \Sigma_p'$, then we say that the right derivative of $\gamma$ at $0$ (or the tangent vector) is $(\xi,1)\in T_p$ or $\xi \in \Sigma_p$. Briefly, $\gamma^+(0)=\xi$ or $\gamma'(0)=\xi$.
If geodesic $\gamma(t)$ is not unit speed, denote $|\gamma'(0)|:=\lim_{t\to 0} \frac{|p\gamma(t)|}{t}$, the right derivative is $(\xi, |\gamma'(0)|)$. Briefly, $\gamma^+(0)=(\xi, |\gamma'(0)|)$ or $\gamma'(0)=(\xi, |\gamma'(0)|)$. In general, let $\alpha:(-a, b)\to X$ be a curve for $a,b\gs 0$ with $\alpha(0)=p$. Let $\gamma_t$ be a geodesic connecting $p$ and $\alpha(t)$.
\begin{defn}
 We say that $v\in T_p$ is the right derivative of $\alpha$ at 0, briefly $\alpha^+(0)=v$, if $\lim_{t\to 0^+}(\gamma_t^+(0), \frac{|p\gamma(t)|}{t})=v$. We say that $v \in T_p$ is the left derivative of $\alpha$ at 0, briefly $\alpha^{-}(0)=v$, if $\lim_{t\to 0^-}(\gamma_t^+(0), -\frac{|p\gamma(t)|}{t})=v$
\end{defn}
\begin{prop}
For any $v\in T_pX$, there is a curve $\alpha:[0,\epsilon)\to X$ such that $\alpha^+(0)=v$.
\end{prop}
Considering Proposition \ref{prop:2geo}, see Proposition 5.6.3 of book \cite{Alexgeo2019} for a proof.

\begin{defn}
Let $\lambda \in \bR$. A locally Lipschitz function $f:X\to \bR$ is called $\lambda$-concave, if for each unit speed geodesic $\gamma(t)$, $f\circ \gamma(t)-\frac{\lambda}{2}t^2$ is a concave function. A locally Lipschitz function $f:X \to \bR$ is called semiconcave, if for any $x\in X$, there is a neighborhood $U_x$ and real number $\lambda_x$, such that the restriction $f|_{U_x}$ is $\lambda_x$-concave.
\end{defn}
\begin{remark}
For Alexandrov spaces with curvature bounded below, it's known that (see \cite{petrunin2007semiconcave}) if f is continuous,  for each geodesic $\gamma(t)$, $f\circ \gamma(t)-\frac{\lambda}{2} t^2$ is concave ,then f is locally Lipschitz.
\end{remark}
For GCBA, we have similar property:
\begin{prop}\label{prop:lip}
 Let $f:X \to \bR$ be a continuous function, if for each geodesic $\gamma(t)$, $f\circ \gamma(t)-\frac{\lambda}{2} t^2$ is concave, then f is locally Lipschitz.
\end{prop}
\begin{proof}
For any $p\in X$, there exists $r>0$, such that any geodesic $\gamma$ with $\gamma(0)=x$ can be extended to geodesic $\gamma:(-2r,2r) \to X$.
Without loss of generality, we can assume that f is concave. Otherwise we can add a very concave (Lipschitz) function. If the neighborhood has curvature $\ls 0$, then $-d^2(p,\cdot)$ is $-1$-concave. If has curvature $\ls -k<0$, then $\frac{1}{\sqrt{k}} \cosh (\sqrt{k} \cdot dist_p)$ is $-1$-concave.
There exists $M>0$, such that $|f(x)|\ls M$ on $\overline{B(p,10 r)}$. For $x,y \in B(p,r)$, suppose $f(y)>f(x)$. consider the geodesic $\gamma$ connecting $x$ and $y$ with $\gamma(0)=x$ and $\gamma(|xy|)=y$. Then
$$
f(x)\gs \frac{r}{r+|xy|} f(y)+\frac{|xy|}{r+|xy|}f\circ \gamma(-r).
$$
$$
\begin{array}{ll}
\frac{f(y)-f(x)}{|xy|} &\ls \frac{f(x)-f\circ \gamma(-r)}{r}\\
&\ls \frac{2M}{r}.
\end{array}
$$
\end{proof}

For semiconcave function $f$, the differential $d_p f:T_p \to \bR$ is well defined: For $v\in T_p$, choose a curve $\alpha(t)$ with $\alpha(0)=p$, $\alpha^+(0)=v$,
\[
d_pf(v):= \lim_{t\to 0} \frac{d(f\circ \alpha(t))}{dt}|_{t=0}.
\]
$d_p f:T_p \to \bR$ is a homogeneous function, i.e. for $v=(\xi, t)\in T_p$, $d_p f(v)=t d_p f(\xi)$.
$d_p f:T_p \to \bR$ is a Lipschitz, concave function (See e.g. Section 7 of \cite{lytchakopen}). There is a unique $\xi_{\max} \in \Sigma_p$ such that $d_p f(\xi_{max})=\sup_{\eta\in \Sigma_p} d_p f(\eta)$. Denote $\nabla_p f:=d_p f(\xi_{max})\xi_{max}$, called the gradient at p. We have $d_p f(w)\ls \langle \nabla_p f, w \rangle$.  Gradient is lower semicontinuous, i.e. $|\nabla_p f| \ls \liminf\limits_{i\to \infty} |\nabla_{x_i} f|$ for $x_i$ converging to p.
\begin{defn}
Let $f:X \to \bR$ be a semiconcave function. A curve $\alpha(t)$ is called a $f$-gradient curve if for any t,
\[
\alpha^+(t)=\nabla_{\alpha(t)} f.
\]
\end{defn}
If $\alpha(t)$ is a $f$-gradient curve, then $\frac{d f\circ \alpha(t)}{dt}=|\nabla_{\alpha(t)} f|^2$. For the three types of spaces, we have the existence and uniqueness of gradient curves:
\begin{prop}[Theorem 1.7 of \cite{lytchakopen}]$\ $
Let $f:X \to \bR$ be a semiconcave function. For each $p\in X$, there exists a unique gradient curve starts at $p$. The $f$-gradient flow is 1-Lipschitz if $f$ is concave.
\end{prop}
Here the $f$-gradient flow is defined to be the one parameter family of maps
\[
\Phi_f^t: X\to X, \qquad \Phi_f^t(p)=\alpha_p(t),
\]
where $t\gs 0$ and $\alpha_p(t)$ is the $f$-gradient curve which starts at $p$.

Let $x,z$ be two points in $X$ connected by a geodesic $\gamma$. Let $\mu$ and $\nu$ be two Lipschitz curves starting at $x$ (respectively, at $z$) and $\mu^+(0)=v\in T_x$ (respectively, $\nu^+\in T_z$). Let $\gamma^+\in T_x$ (respectively, $\gamma^{-1}\in T_z$ be the original (respectively, the terminal ) direction of $\gamma$. Then we have the following first variation inequality:
\begin{lem}[Theorem 1.2 of \cite{lytchakdiff}]$\ $
Under the above conditions, if the function $l(t)=d(\mu(t),\nu(t))$ is differentiable at $0$ from the right, then
\begin{equation}\label{in:first}
l^+(0)\leqslant -\langle \gamma^+, v\rangle -\langle \gamma^{-}, w \rangle
\end{equation}
\end{lem}

\section{Backward gradient curve} $\ $
For two points $p,q\in X$, $\uparrow_p^q$ denotes $\gamma^+$ for a unit speed geodesic $\gamma$ connecting $p$ (initial point) and $q$.

\begin{prop}\label{prop:op}
Let $X$ be one of: n dimensional SCBBL without boundary; GCBA; $C^{\alpha}$ Riemannian manifolds without boundary.
Then for any $p\in X$, there exists a vector $\xi_{min}\in \Sigma_p X$, such that
\begin{equation}\label{in:max}
|\nabla_ p f|\ls -d_p f(\xi_{min}).
\end{equation}
\end{prop}
\begin{proof}
If $d_p f(\xi)\ls 0$ for any $\xi \in \Sigma_p$, then we just choose $\xi_{min}=o_p\in T_p X$. Now assume that $d_pf(\xi_{max})>0$. If  $X$ is an Alexandrov space with curvature bounded below locally, let $\xi_{min}\in \Sigma_p$ be a minimum point of $d_pf|_{\Sigma_p}$, then we have (\ref{in:max}), see e.g. Lemma 1.3.7 of \cite{petrunin2007semiconcave}. If $X$ is a $C^{\alpha}$ Riemannian manifold then $T_p X=\bR^n$. Since $d_pf$ is concave on $T_p X$, we have $d_p f(\nabla_p f)+d_pf(-\nabla_p f)\ls 2d_pf(o_p)=0$, then $|\nabla_p f|=d_pf(\frac{\nabla_p f}{|\nabla_pf|})\ls -d_pf(-\frac{\nabla_p f}{|\nabla_pf|})$. Let $\xi_{min}=-\frac{\nabla_p f}{|\nabla_pf|}$. If $(X,d)$ is a locally compact, geodesically complete space with curvature bounded above, let $B(p,2r)$ be a normal ball. For $\xi_{\max}\in\Sigma_p$, by e.g. 2.2.8 of \cite{CAT}, there exists a point $x\in B(p,2r)$ with $|px|=r$, for the (unique) geodesic $\gamma$ connecting $p$ and $x$, $\gamma^+=\xi_{max}$. Now extend $\gamma$ through $p$, we get another direction $\xi_{min}=-\xi_{\max}$. Then $d_pf(\xi_{max})\ls -d_pf(\xi_{min})$.
\end{proof}
The following is an immediate corollary,
\begin{cor}
Let $X$ be one of: Compact n dimensional SCBBL without boundary; Compact, locally geodesically complete spaces with curvature bounded above; Compact $C^{\alpha}$ Riemannian manifold without boundary.
Then any locally Lipschitz, concave function $f:X\to \bR$ defined on $X$ is a constant.
\end{cor}
\begin{proof}
If $f$ is not a constant. Suppose $f$ attains its minimum at $p\in X$, attains its maximum at $q\in X$. Then $d_pf(\xi)\ls 0$ for any $\xi\in \Sigma_p$. Since $|\nabla_ p f|\ls -d_p f(\xi_{min})$, we have $|\nabla_p f|=0$. On the other hand, choose a geodesic $\gamma$ connecting $p$ and $q$. Since $f\circ \gamma$ is a concave function, $d_p f(\gamma^+(0))\gs \frac{f(q)-f(p)}{|pq|}>0$, contradiction. So $f$ must be a constant.
\end{proof}
\begin{defn}
A backward gradient curve will be denoted by $\beta(t)$, $t\in [0,\infty)$, i.e. $\alpha(t)=\beta(t_0-t)$ is a gradient curve for $t<t_0$.
\[
\beta^{-}(t)=\lim_{\epsilon \to 0^+} \frac{|\beta(t)\beta(t-\epsilon)|}{\epsilon}\uparrow_{\beta(t)}^{\beta(t-\epsilon)}=\nabla_{\beta(t)}f
\]
Backward gradient curve $\beta(t)$ reparametrized by arc length, will be denoted by $\bar{\beta}(t)$, i.e. $\bar{\beta}^{-}(t)=\frac{\nabla_{\beta(t)}f}{|\nabla_{\beta(t)}f|}$.
\end{defn}
If $\nabla_p f=0$, then the curve $\bar{\beta}_k(t)\equiv p$ is a backward gradient curve. For $\nabla_p f\neq 0$, if $(X,d)$ is an n dimensional Alexandrov space with curvature bounded below, by \cite{GroPe1993}, for sufficiently small $t\gs 0$, there exists $q\in X$ such that $\Phi_f^t(q)=p$, i.e. there exists a backward gradient curve $\beta(t)$ beginning from $p$. Their proof relies on homology. In an unpublished notes \cite{hczhang} by HuiChun Zhang, he shows a direct proof. Following the lines of his proof, we can get the following property on the existence of backward gradient curves:
\begin{thm}
Under the condition of Proposition \ref{prop:op},
Let $f$ be a semiconcave function. If $\nabla_p f\neq 0$, then there is $\epsilon>0$, such that the backward gradient curve $\beta(t)$ can be defined on $t\in [0,\epsilon)$ with $\beta(0)=p$ for some $\epsilon>0$.
\end{thm}
\begin{proof}
Step 1, construct broken backward curve $\bar{\beta}_k(t)$, parametrized by arc length.

For integer $k$ with $\frac1k<\frac{c_0}{100}$, denote $p_{0,k}=p$, choose $p_{1,k}$ with $|pp_{1,k}|<\frac1k$, such that $d_pf(\uparrow_p^{p_{1,k}})<d_pf(\xi_{min})+\frac{1}{k}$ and
\[
\begin{array}{ll}
f(p_{1,k})-f(p)
&\ls d_pf(\uparrow_p^{p_{1,k}})|pp_{1,k}|+\frac{\lambda}{2}|pp_{1,k}|^2\\
&< (d_pf(\xi_{min})+\frac{1}{k})|pp_{1,k}|
\end{array}
\]
Since $-d_p f(\xi_{min})\geqslant |\nabla_p f|$, we have
\[
f(p_{1,k})-f(p)<(-|\nabla_p f|+\frac{1}{k})|pp_{1,k}|.
\]

In general, for $i\geqslant 1$, we choose $p_{i,k}$ such that $|p_{i,k}p_{i+1,k}|\leqslant \frac{1}{k}$ and
\[
\begin{array}{ll}
f(p_{i+1,k})-f(p_{i,k})&<(d_{p_{i,k}}f(\xi_{min})+\frac{1}{k})|p_{i,k}p_{i+1,k}|\\
&\leqslant (-|\nabla_{p_{i,k}} f|+\frac{1}{k})|p_{i,k}p_{i+1,k}|.
\end{array}
\]
For any $j>i$, we have
\begin{equation}\label{in:broken}
\begin{array}{ll}
f(p_{i,k})-f(p_{j,k})&=(f(p_{i,k})-f(p_{i+1,k})+...+(f(p_{j-1,k})-f(p_{j,k}))\\
&>(|\nabla_{p_{i,k}}f|-\frac{1}{k}) |p_{i,k} p_{i+1,k}|+...+(|\nabla_{p_{j-1,k}}f|-\frac{1}{k})|p_{j-1,k} p_{j,k}|\\
&\gs \inf\limits_{i\leqslant l\leqslant j-1} (|\nabla_{p_{l,k}}f|-\frac{1}{k})(|p_{i,k}p_{i+1,k}|+...+|p_{j-1,k}p_{j,k}|)\\
&\gs \inf\limits_{i\leqslant l\leqslant j-1} (|\nabla_{p_{l,k}}f|-\frac{1}{k})|p_{i,k}p_{j,k}|.
\end{array}
\end{equation}

If $sup|pp_{i,k}|<\frac{r}{2}$, then there exists a subsequence $\{p_{i_j,k}\}$, $p_{i_j,k}$ converge to some $p_{\infty,k}$ as $j\to \infty$. Since $f(p_{i,k})$ is strictly decreasing, $\lim\limits_{i\to \infty} f(p_{i,k})$ exists, and equals to $\lim\limits_{j\to \infty} f(p_{i_j,k})=f(p_{\infty,k})$. So $f(p_{\infty,k})=\inf_i f(p_{i,k})$. By (\ref{in:broken}), for $i,j$ large,
\[
|p_{i,k}p_{j,k}|\ls \frac{f(p_{i,k})-f(p_{j,k})}{c_0-\frac1k}.
\]
Then $p_{i,k}$ converge to $p_{\infty,k}$, if $i\to \infty$. Let $\bar{\beta}_k(t)$ be broken geodesics made up of short segments $pp_{1,k}, p_{1,k}p_{2,k}...$, parametrized by arc length. By (\ref{in:broken}),
\[
|pp_{1,k}|+|p_{1,k}p_{2,k}|+...+|p_{i-1,k}p_{i,k}|\ls \frac{f(p)-f(p_{i,k})}{c_0-\frac1k}.
\]
Let $i\to \infty$, we have
\[
\sum_{i=0}^{\infty}|p_{i,k}p_{i+1,k}|\ls \frac{f(p)-f(p_{\infty,k})}{c_0-\frac1k}.
\]
Denote $t_0=\sum_{i=0}^{\infty}|p_{i,k}p_{i+1,k}|$, then $\bar{\beta}_k(t)$ can be defined on $[0,t_0]$ with $\bar{\beta}_k(t_0)=p_{\infty,k}$. And we have
\begin{equation}
\begin{array}{ll}
f(p_{i,k})-f(p_{\infty,k})
&=\lim\limits_{j\to \infty} [f(p_{i,k})- f(p_{j,k})]\\
&>\limsup\limits_{j\to \infty}\inf\limits_{i\ls l\ls j-1}(|\nabla_{p_{l,k}}f|-\frac1k)|p_ip_{j,k}|\\
&\gs \inf\limits_{i\leqslant l<\infty} (|\nabla_{p_{l,k}}f|-\frac{1}{k})|p_{i,k}p_{\infty,k}|\\
\end{array}
\end{equation}
Then beginning from $p_{\infty,k}$, we can construct points that satisfy (\ref{in:broken}). By repeating the argument, we claim that $\sup_t |p\bar{\beta}_k(t)|\gs \frac{r}{2}$, where $t$ is among the interval $\bar{\beta}_k(t)$ can be defined. If $\bar{\beta}_k(t)$ can be defined on $[0,\infty)$. By (\ref{in:broken}), $\lim\limits_{t\to \infty} f(\bar{\beta}_k(t))=-\infty$. Since $f$ is bounded on $\overline{B(p,\frac{r}{2})}$, $\sup_t |p\bar{\beta}_k(t)|\gs \frac{r}{2}$. If $\sup_t |p\bar{\beta}_k(t)|<\frac{r}{2}$, suppose $[0,t_1)$ is the maximal interval that $\bar{\beta}_k(t)$ can be defined on. By the above arguments, $\bar{\beta}_k(t)$ can be defined on $[0,t_1]$. Hence there exists $\epsilon>0$, $\bar{\beta}_k(t)$ can be defined on $[0,t_1+\epsilon)$. So we must have $\sup_t |p\bar{\beta}_k(t)|\gs \frac{r}{2}$.

Step 2, Show a limit of $\bar{\beta}_k(t)$ as $k\to \infty$ is a backward gradient curve parametrized by arc length.

Suppose that $|\nabla_p f|>c_0>0$, then there exists $r>0$, such that for any $q\in B(p,r)$, $|\nabla_q f|\gs c_0>0$.
Recall that $\bar{\beta}_k(t)$ is the broken geodesic made up of short segments $pp_{1,k}, p_{1,k}p_{2,k}...$. For $t\in [0,\frac{r}{2}]$, suppose that $\bar{\beta}_k(t)$ subconverge to a curve $\bar{\beta}(t)$. Then for $t_2>t_1\gs 0$, $|\bar{\beta}(t_2)-\bar{\beta}(t_1)| \ls t_2-t_1$.
Denote $\uparrow_{\bar{\beta}(t)}^{\bar{\beta}(t-\epsilon)}\in \Sigma_p$ the direction of a geodesic connecting $\bar{\beta}(t)$ and $\bar{\beta}(t-\epsilon)$.

We claim that for every $t>0$, $\lim\limits_{\epsilon \to 0^+} \frac{|\bar{\beta}(t-\epsilon)\bar{\beta}(t)|} {\epsilon}=1$ and $\bar{\beta}^{-}(t):=\lim\limits_{\epsilon\to 0^+} \frac{|\bar{\beta}(t-\epsilon)\bar{\beta}(t)|}{\epsilon}\uparrow_{\bar{\beta}(t)}^{\bar{\beta}(t-\epsilon)}$ exists and equals to $\frac{\nabla_{\bar{\beta}(t)} f}{|\nabla_{\bar{\beta}(t)} f|}$.

In fact, for $0\ls t_1\ls t_2$, since $|p_{i,k}p_{i+1,k}|<\frac1k \to 0$, we can
suppose that $p_{i(k),k}$ converge to $\bar{\beta}(t_1)$, $p_{j(k),k}$ converge to $\bar{\beta}(t_2)$ as $k\to \infty$. Since $|p_{i(k),k} p_{i(k)+1,k}|+...+|p_{j(k)-1,k}, p_{j(k),k}|=t_2-t_1$, by (\ref{in:broken}), we have
\begin{equation}\label{in:t1t2}
\begin{array}{ll}
f(\bar{\beta}(t_1))-f(\bar{\beta}(t_2))&\gs \inf_{t_1\ls t<t_2} |\nabla_{\bar{\beta}(t)} f|(t_2-t_1)\\
&\gs \inf_{t_1\ls t<t_2} |\nabla_{\bar{\beta}(t)} f||\bar{\beta}(t_1)\bar{\beta}(t_2)|.
\end{array}
\end{equation}
For $t_0>0$ and $\epsilon>0$
\[
\frac{f(\bar{\beta}(t_0-\epsilon))-f(\bar{\beta}(t_0))}{|\bar{\beta}(t_0-\epsilon)\bar{\beta}(t_0)|}\gs \inf_{t_0-\epsilon\ls t<t_0}|\nabla_{\bar{\beta}(t)}f|.
\]
Let $\epsilon\to 0$, by the l.s.c. of gradient, we have
\begin{equation}\label{in:1}
\liminf_{\epsilon\to 0} \frac{f(\bar{\beta}(t_0-\epsilon))-f(\bar{\beta}(t_0))}{|\bar{\beta}(t_0-\epsilon)\bar{\beta}(t_0)|}\gs |\nabla_{\bar{\beta}(t_0)} f|.
\end{equation}

On the other hand, we have
\begin{equation}\label{in:2}
f(\bar{\beta}(t_0-\epsilon))-f(\bar{\beta}(t_0))\ls d_{\bar{\beta}(t_0)}f(\uparrow_{\bar{\beta}(t_0)}^{\bar{\beta}(t_0-\epsilon)})
|\bar{\beta}(t_0-\epsilon)\bar{\beta}(t_0)|+\frac{\lambda}{2}|\bar{\beta}(t_0-\epsilon)\bar{\beta}(t_0)|^2.
\end{equation}
So
\begin{equation}\label{in:3}
\limsup_{\epsilon\to 0} \frac{f(\bar{\beta}(t_0-\epsilon))-f(\bar{\beta}(t_0))}{|\bar{\beta}(t_0-\epsilon)\bar{\beta}(t_0)|}
\ls \limsup_{\epsilon\to 0} d_{\bar{\beta}(t_0)}f(\uparrow_{\bar{\beta}(t_0)}^{\bar{\beta}(t_0-\epsilon)}) \ls |\nabla_{\bar{\beta}(t_0)}f|.
\end{equation}
By combining inequalities (\ref{in:1}),(\ref{in:2}) and (\ref{in:3}), we get that
\[
\lim_{\epsilon\to 0} \frac{f(\bar{\beta}(t_0-\epsilon))-f(\bar{\beta}(t_0))}
{|\bar{\beta}(t_0-\epsilon)\bar{\beta}(t_0)|}
\]
exists and equals to $|\nabla_{\bar{\beta}(t_0)}f|$. By (\ref{in:t1t2}), we have $\frac{|\bar{\beta}(t-\epsilon)\bar{\beta}(t)|} {\epsilon}=1$. By (\ref{in:2}),
\[
\begin{array}{ll}
\liminf_{\epsilon\to 0} d_{\bar{\beta}(t_0)}f(\uparrow_{\bar{\beta}(t_0)}^{\bar{\beta}(t_0-\epsilon)})
&\gs |\nabla_{\bar{\beta}(t_0)}f| \\
&=d_{\bar{\beta}(t_0)} f
(\frac{\nabla_{\bar{\beta}(t_0)} f} {|\nabla_{\bar{\beta}(t_0)}f|}).
\end{array}
\]
Then we have $\lim\limits_{\epsilon \to 0} \uparrow_{\bar{\beta}(t)}^{\bar{\beta}(t-\epsilon)}$ exists and equals to $\frac{\nabla_{\bar{\beta}(t)} f}{|\nabla_{\bar{\beta}(t)} f|}$.

So $\bar{\beta}^{-}(t):=\lim\limits_{\epsilon\to 0} \frac{|\bar{\beta}(t-\epsilon)\bar{\beta}(t)|} {\epsilon} \uparrow_{\bar{\beta}(t)}^{\bar{\beta}(t-\epsilon)}$ exists and equals to $\frac{\nabla_{\bar{\beta}(t)} f}{|\nabla_{\bar{\beta}(t)} f|}$.
Define $\beta(t)=\bar{\beta}(s)$, where $t=\int_0^s \frac{1}{|\nabla_{\bar{\beta}(l)} f|} dl$. Then for $t>0$,
\[
\begin{array}{ll}
\beta^{-}(t)&=\frac{\nabla_{\bar{\beta}(s)}f}{|\nabla_{\bar{\beta}(s)|}f|}\cdot |\nabla_{\bar{\beta}(s)}f|\\
&=\nabla_{\beta(t)} f,
\end{array}
\]
then $\beta(t)$ is a backward gradient flow.
\end{proof}

\textbf{In the rest of this section, we always assume that $(X,d)$ is one of:}
complete, n dimensional SCBBL without boundary; Complete GCBA; $C^{1,\alpha}$ differential manifolds with a $C^{\alpha}$-Riemannian metric $g$, without boundary.

\begin{remark}
If a backward gradient curve $\beta(t)$ can be defined on $t\in [0,t_0]$, then $|\nabla_{\beta(t_0) }f|\neq 0$. So there exists $\epsilon>0$, such that $\beta(t)$ can be defined on $t\in [0, t_0+\epsilon)$. So if $\nabla_p f\neq 0$, there exists a backward gradient curve $\beta(t)$ with $\beta(0)=p$, $\beta(t)$ can be defined on $t\in [0,\infty)$ or $[0,t_0)$ for some $t_0<\infty$. For example, let $o$ be the vertex of a cone, $p\neq o$. Consider the function $f(x)=|ox|$, then the backward gradient curve can be defined on $t\in [0,|op|)$. If $f(x)=|ox|^2$, then the backward gradient curve can be defined on $[0,\infty)$.
\end{remark}

\begin{thm}\label{thm:out}
Suppose that $f$ is semiconcave on $B(p,R)$. If $\nabla_q f\neq 0$ for any $q\in B(p,R)$, then for any $r<R$, there exists $t(r)>0$ such that $\beta(t) \notin \overline{B(p,r)}$ for $t> t(r)$.
\end{thm}

\begin{proof}
By the l.s.c. of gradient, there exists $c>0$, such that $|\nabla_x f|\geqslant c>0$ for any $x\in \overline{B(p,r)}$. It's sufficient to prove that there exists $t_0(r)>0$ such that $\bar{\beta}(t)\notin \overline{B(p,r)}$ if $t>t_0(r)$.
There exists $M>0$, such that $|f(x)|\leqslant M$ for any $x\in \overline{B(p,r)}$.  If $\bar{\beta}(t)\in \overline{B(p,r)}$, then
\[
\begin{array}{ll}
2M
&\geqslant f(p)-f(\bar{\beta}(t))\\
&=\int_0^t |\nabla_{\bar{\beta}(s)} f| ds\geqslant c t.
\end{array}
\]
Denote $[0,t_{max})$ the maximal interval such that $\bar{\beta}(t)$ can be defined on. Where $t_{max}$ may be finite or infinite. If $t_{max}>\frac{2M}{c}$, then for $t>\frac{2M}{c}$, $\bar{\beta}(t)\notin \overline{B(p,r)}$. If $t_{max}\leqslant \frac{2M}{c}$, we claim that there exists $t_0$ such that if $t>t_0$ then $\bar{\beta}(t)\notin \overline{B(p,r)}$. In fact, if not, then there exists $t_i\to t_{\max}$, such that $\beta(t_i)\in \overline{B(p,r)}$. For $t_2>t_1\gs 0$,
\[
|\bar{\beta}(t_1)\bar{\beta}(t_2)|\ls t_2-t_1,
\]
so $\lim_{\epsilon\to 0} \bar{\beta}(t_{max}-\epsilon)$ exists, denoted by $q$, $q\in \overline{B(p,r)}$. By the l.s.c. of gradient, for any $\eta>0$, there exists $\delta>0$, such that if $0<\epsilon<\delta$, then $|\nabla_{\bar{\beta}(t_{max}-\epsilon)}f|>|\nabla_q f|-\eta$. So for $0<t<\epsilon<\delta$,
\[
f(\bar{\beta}(t_{max}-\epsilon))-f(\bar{\beta}(t_{max}-t))=\int_{t_{max}-\epsilon}^{t_{max}-t} |\nabla_{\bar{\beta}(s)} f|ds\gs (\epsilon-t)(|\nabla_q f|-\eta).
\]
Let $t\to 0$, we get
\[
f(\bar{\beta}(t_{max}-\epsilon))-f(q)\gs \epsilon (|\nabla_q f|-\eta).
\]
Then
\begin{equation}\label{in:q1}
\liminf\limits_{\epsilon\to 0} \frac{f(\bar{\beta}(t_{max}-\epsilon))-f(q)}{\epsilon}
\gs |\nabla_q f|.
\end{equation}
On the other hand
\begin{equation}\label{in:onthe}
f(\bar{\beta}(t_{max}-\epsilon))-f(q)\ls d_q f(\uparrow_q^{\bar{\beta}(t_{max}-\epsilon)})|q\bar{\beta}(t_{max}-\epsilon)|
+\frac{\lambda}{2}|q\bar{\beta}(t_{max}-\epsilon)|^2.
\end{equation}
Since $|q\bar{\beta}(t_{max}-\epsilon)|\ls \epsilon$.
\begin{equation}\label{in:q2}
\begin{array}{ll}
\limsup\limits_{\epsilon \to 0} \frac{f(\bar{\beta}(t_{max}-\epsilon)-f(q)}{\epsilon}&\ls \limsup\limits_{\epsilon\to 0} d_q f(\frac{|q\bar{\beta}(t_{max}-\epsilon)|}{\epsilon}\uparrow_q^{\bar{\beta}(t_{max}-\epsilon)})\\
&\ls |\nabla_q f|.
\end{array}
\end{equation}
By combining (\ref{in:q1}) with (\ref{in:q2}), we have
\[
\lim_{\epsilon\to 0}\frac{f(\bar{\beta}(t_{max}-\epsilon))-f(q)}{\epsilon}=|\nabla_q f|,
\]
By (\ref{in:onthe}) and the above equality,
\[
\liminf_{\epsilon \to 0} d_q f(\frac{|q\bar{\beta}(t_{max}-\epsilon)|}{\epsilon}\uparrow_q^{\bar{\beta}(t_{max}-\epsilon)})\gs |\nabla_q f|.
\]
Then
\[
\lim_{\epsilon \to 0} \frac{|q\bar{\beta}(t_{max}-\epsilon)|}{\epsilon}\uparrow_q^{\bar{\beta}(t_{max}-\epsilon)}
=\frac{\nabla_q f}{|\nabla_q f|},
\]
and $\lim\limits_{\epsilon \to 0} \frac{|q\bar{\beta}(t_{max}-\epsilon)|}{\epsilon}=1$. We define $\bar{\beta}(t_{max}):=q$, then
\[
\bar{\beta}^{-}(t_{max})=\frac{\nabla_q f}{|\nabla_q f|}.
\]
So $\bar{\beta}(t)$ can be defined on $t\in[0,t_{max}]$, hence can be defined on $t\in [0,t_{max}+\epsilon)$ for some $\epsilon>0$, contradiction.
\end{proof}

The following proposition is known to experts, see e.g. Lemma 2.14 of \cite{petrunin2007semiconcave}. For readers' convenience, we list a proof here.
\begin{prop}
Let $f:X\to \bR$ be a concave function. $\beta(t)$ be the backward gradient curve with $\beta(0)=p$. If $f(q)>f\circ \beta(t_0)$, then $|q\beta(t)|$  is strictly increasing for $t\geqslant t_0$. In particular, $|p\beta(t)|$ is strictly increasing for $t\gs 0$.
\end{prop}
\begin{proof}
Since $|q\bar{\beta}(t)|$ is Lipschitz, it's differentiable for a.e. t. For $t\gs t_0$, $f(\bar{\beta}(t))\ls f(\bar{\beta}(t_0))<f(q)$. Suppose $|q\bar{\beta}(t)|$ is differentiable at $t(\gs t_0)$, let $\gamma(s)$ be a unit speed geodesic connecting $\beta(t)$ and $q$. Since $f\circ \gamma(s)$ is concave, we have that
\[
d_{\bar{\beta}(t)} f(\uparrow_{\bar{\beta}(t)}^q)\gs \frac{f(q)-f(\bar{\beta}(t))}{|q\bar{\beta}(t)|}>0.
\]
Then $\langle \nabla_{\bar{\beta}(t)}f, \uparrow_{\bar{\beta}(t)}^q \rangle \gs d_{\bar{\beta}(t)} f(\uparrow_{\bar{\beta}(t)}^q)>0$. On the other hand, by the first variation inequality (\ref{in:first})

\[
\lim_{\epsilon\to 0} \frac{|q\bar{\beta}(t-\epsilon)|-|q\bar{\beta}(t)|}{\epsilon}\leqslant -\langle \nabla_{\bar{\beta}(t)} f, \uparrow_{\bar{\beta}(t)}^p \rangle<0.
\]
So $|q\bar{\beta}(t)|$ (hence $|q\beta(t)|$) is strictly increasing for $t\gs t_0$.
\end{proof}

The following proposition is known to experts, see e.g. Lemma 2.1.3 of \cite{petrunin2007semiconcave}.
\begin{prop}\label{prop:nabla}
If $f:X\to \bR$ is a concave function, and $\alpha(t)$ is a gradient curve. let $\bar{\alpha}(t)$ be the reparametrization of $\alpha(t)$ by arclength. Then $f\circ \bar{\alpha}(t)$ is a concave function. Hence $|\nabla_{\alpha(t)} f|$ is decreasing along $\alpha(t)$, $|\nabla_{\beta(t)} f|$ is increasing along $\beta(t)$.
\end{prop}

For the theorems of the introduction, now we can prove the particular case when $f:X \to \bR$ is a Lipschitz, concave function.
\begin{thm}
If $X$ admits a non-constant, Lipschitz, concave function, then $\cH^n(X)=\infty$.
\end{thm}
\begin{proof}
Denote $S=\{x: f(x)=\sup_{y\in X}f(y)\}$, $S$ may be an empty set. For $p\notin S$, choose $q\in X$ such that $f(q)>f(p)$. Then $f(q)-f(p)\ls d_pf(\uparrow_p^q)|pq|$, then $|\nabla_p f|>\frac{f(q)-f(p)}{|pq|}$. So any point of $X\backslash S$ is noncritical. Since $f$ is Lipschitz, there exists $L>0$, such that $|f(y)-f(x)|\ls L|xy|$. Then $|\nabla_p f|\ls L$ for any $p\in X$.

Choose $p\in X\backslash S$, and $r$ small such that $|\nabla_q f|>c_0>0$ for any $q\in B(p,r)$. By proposition \ref{prop:nabla}, $|\nabla_{\beta_q(t)}f|\gs |\nabla_q f|>c_0$. Let $\cH^n(B(p,r))=v_0>0$, consider the backward gradient flow, it's distance expanding, we have $\cH^n(\Phi_f^{-t}(B(p,r)))\gs \cH^n(B(p,r))$.

We claim that if $t_2-t_1>\frac{10M}{c_0^2}$, then $\Phi_f^{-t_1}(B(p,r))\cap \Phi_f^{-t_2}(B(p,r))=\emptyset$.

In fact, if the claim is false, choose $x=\Phi_f^{-t_1}(B(p,r))\cap \Phi_f^{-t_2}(B(p,r))$, then $y=\Phi_f^{t_1}(x)\in B(p,r)$, $z=\Phi_f^{t_2}(x)\in B(p,r)$. So $z=\Phi_f^{t_2-t_1}(y)$, i.e. the gradient curve $\alpha_y(t)$ beginning from $y$ satisfies $\alpha_y(t_2-t_1)=z$. Then
$$
\begin{array}{ll}
2M&\gs f(z)-f(y)\\
&=\int_0^{t_2-t_1}|\nabla_{\alpha_y(s)} f|^2 ds \\
&\gs c_0^2(t_2-t_1)\\
&\gs 10M,
\end{array}
$$
contradiction!

For any $q\in B(p,r)$, since the length of backward gradient curve $\beta_q(t)$ is $\int_0^{t_{max}} |\nabla_{\beta_q(s)}f| ds=\infty$ and $|\nabla_ x f| \ls L$ for any $x\in X$, $t_{max}=\infty$.
Now we choose an increasing sequence $0=t_0<t_1<...t_i<t_{i+1}...$ with $t_i\to \infty$, such that $t_{i+1}-t_i \gs \frac{10M}{c_0^2}$, then $\{\Phi_f^{-t_i}(B(p,r))\}_i$ are pairwise disjoint.
\[
\cH^n(X)\gs \sum_i \cH^n(\Phi_f^{-t_i}(B(p,r)))=\infty.
\]
\end{proof}
Since the distance function is convex and 1-Lipschitz in a Hadamard space, we have the following corollary:
\begin{cor}
Let $X$ be a complete, n dimensional, locally geodesically complete Hadamard space, then $\cH^n(X)=\infty$.
\end{cor}

\begin{remark}\label{remark}
For a locally Lipschitz, concave function $f:X \to \bR$. For $q\in B(p,r)$, $|\nabla_{\beta_q(t)} f|$ is increasing w.r.t. $t$. It's possible that for any $q\in B(p,r)$, $|\nabla_{\beta_q(t)} f|$ is increasing rapidly and $\int_0^{t_{max}} |\nabla_{\beta_q(s)} f| ds =\infty$ for some $t_{max}<\infty$. i.e. the backward gradient curves go to infinity in finite time. Then there are no sequence $\{t_i\}_{i=1}^{\infty}$ with $t_{i+1}-t_i \gs \frac{10 M}{c_0^2}$ and $t_i<t_{max}$. So we should consider another method, see the final section.
\end{remark}

\section{Coarea inequalities} $\ $
Let $X$ be a metric space and $A\subset X$ be a subset. Let $d$ be a nonnegative real number. For an $\epsilon>0$ define $\mu_{d,\epsilon}$ by
\[
\mu_{d,\epsilon}(X)=\inf \left \{\sum_i (diam S_i)^d: A\subset \cup_i S_i, diam(S_i)<\epsilon \text{ for all i } \right \}.
\]
The infimum is taken over all finite or countable covering $\{S_i\}_{i\in I}\subset X$ of $A$ by sets of diameter $<\epsilon$. If no such covering exists, then the infimum is $+\infty$.

The $d$-dimensional Hausdorff measure of $A$ is defined by the formula
\[
\mu_d(X)=C_{\cH}(d)\cdot \lim_{\epsilon\to 0} \mu_{d,\epsilon}(X).
\]
Where $C_{\cH}(d)$ is a positive normalization constant. Denote
\[
Lip f(x)=\limsup_{y\to x}\frac{|f(y)-f(x)|}{|xy|},
\]
For a Borel subset $B\subset X$, denote
\[
Lip(f,B)=\sup_{x,y\in B}\frac{|f(y)-f(x)|}{diam B}.
\]
let $x,y,x_0\in B$, we have
\begin{equation}\label{in:22}
\begin{array}{ll}
\frac{|f(y)-f(x)|}{diam B}
&\ls \frac{|f(y)-f(x_0)|+|f(x_0)-f(x)|}{diam B}\\
&\ls 2\sup_{z \in B} \frac{|f(z)-f(x_0)|}{diam B}\\
&\ls 2Lip f(x_0).
\end{array}
\end{equation}
Then
\begin{equation}\label{in:LipfB}
Lip(f,B)\ls 2Lip f(x_0).
\end{equation}

\textbf{In the rest of this section, we always assume that $(X,d)$ is one of the following three types of spaces}:
Complete, n dimensional SCBBL without boundary; Complete, n dimensional GCBA; $C^{1,\alpha}$ manifolds with a $C^{\alpha}$ Riemannian metric, without boundary.

The following Eilenberg inequality is known to experts, see e.g. Theorem 2.10.25 of \cite{federer1969}, proposition 5.1 of \cite{Ambro2002}. For readers' convenience, we list a proof here.

\begin{prop}\label{prop:co1}
Let $f$ be a locally Lipschitz function, for any $\cH^n$-measurable subset $A$,
\[
\int_{-\infty}^{\infty}\cH^{n-1}(A\cap f^{-1}(t))dt \leqslant C(n) \sup_{x\in A} Lipf(x) \cH^n(A)
\]
\end{prop}
\begin{proof}
 If $\cH^n(A)=\infty$, then the inequality holds. So assume that $\cH^n(A)<\infty$. There exist balls $\{B_{j,i}\}_i$ with $diam B_{j,i} <\frac1j$, such that $A\subset \cup_i B_{j,i}$ and $\sum_i C_{\cH}(n) (diam B_{j,i})^n <\cH^n(A)+\frac1j$. Since
\[
\cH^{n-1}_{\frac1j}(A\cap f^{-1}(t)) \leqslant \sum_i C_{\cH}(n-1) (diam B_{j,i})^{n-1}\chi_{f(B_{j,i})}(t),
\]
where $\chi_{f(B_{j,i})}(t)=1$ if $t\in f(B_{j,i})$, $\chi_{f(B_{j,i})}(t)=0$ if $t\notin f(B_{j,i})$. Then

\begin{equation}\label{in:chi}
\begin{array}{ll}
\int_{-\infty}^{\infty} \chi_{f(B_{j,i})}(t) dt &\leqslant \sup_{x,y\in B_{j,i}}|f(y)-f(x)|\\
&\ls \sup_{x,y\in B_{j,i}}\frac{|f(y)-f(x)|}{diam B_{j,i}}diam B_{j,i}.
\end{array}
\end{equation}

\[
\begin{array}{ll}
\int_{-\infty}^{\infty}\cH^{n-1}(A\cap f^{-1}(t))dt
&=\int_{-\infty}^{\infty} \lim\limits_{j \to \infty}\cH^{n-1}_{\frac1j}  (A\cap f^{-1}(t)) dt\\
&\leqslant \int_{-\infty}^{\infty} \lim\limits_{j\to \infty} \sum_i  C_{\cH}(n-1)(diam B_{j,i})^{n-1}\chi_{f(B_{j,i})}(t) dt\\
&\leqslant \liminf\limits_{j\to \infty} \int_{-\infty}^{\infty} \sum_i  C_{\cH}(n-1)(diam B_{j,i})^{n-1}\chi_{f(B_{j,i})}(t) dt\\
&=\liminf\limits_{j\to \infty} \sum_i C_{\cH}(n-1)(diam B_{j,i})^{n-1} \int_{-\infty}^{\infty} \chi_{f(B_{j,i})}(t) dt\\
&\leqslant \frac{C_{\cH}(n-1)}{C_{\cH}(n)} \liminf\limits_{j\to \infty} \sum C_{\cH}(n) \sup\limits_{x,y\in B_{j,i}}\frac{|f(y)-f(x)|}{diam B_{j,i}}(diam B_{j,i})^n\\
&\leqslant 2\sup\limits_{x\in A} Lip f(x) \frac{C_{\cH}(n-1)}{C_{\cH}(n)} \cH^n(A).
\end{array}
\]
\end{proof}

the following property is known to experts, see e.g. 9.1 of \cite{lytchakopen}.
\begin{prop}\label{prop:equal}
For a locally Lipschitz, semiconcave function, for any $x\in X$, $|\nabla_x f|=Lip f(x)$.
\end{prop}

For countably rectifiable subsets metric spaces, the coarea formula holds, see \cite{AmKir2000}. Since we just need a coarea inequality, for readers' convenience, we list a direct proof here. The proof is a small modification of the proof of Proposition 5.1,  \cite{AmMaGi2017}.
\begin{prop}
Let $A\subset X$ be a $\cH^n$-measurable subset, for any locally Lipschitz function $f:X\to \bR$,
\begin{equation}\label{in:co1}
\int_{-\infty}^{\infty} \cH^{n-1} (f^{-1}(t)\cap A) dt \leqslant C(n) \int_A Lip f(x) d \cH^{n-1}.
\end{equation}
\end{prop}
\begin{proof}
Without loss of generality, we may assume that $A$ is bounded. Since we can prove the inequality for $A\cap B(p,R)$, then let $R\to \infty$.

Claim: For any sufficiently small $\epsilon>0$ and any positive integer $j$, we can find disjoint balls $B(x_i,r_i)$ with $diam B_{j,i} \ls 2r_i \ls \frac1j$, for any $r\ls r_i$, we have
\begin{equation}\label{in:reg}
\alpha(n) r^n \ls (1+\epsilon) \cH^n(B(x_i,r))
\end{equation}
and $\cH^n(A \backslash \cup_i B(x_i,r_i))=0$, where $\alpha(n)$ is the volume of the unit ball in $\bR^n$.

In fact, for Alexandrov spaces with curvature bounded below, the set of regular points $R_X\subset X$ has full measure, i.e. $\cH^n(X\backslash R_X)=0$. For any regular point $x$, there exists $r_x>0$, such that there is a $(1\pm \frac{\epsilon}{100n})$ bi Lipschitz homeomorphism from $B(x,r_x)$ onto a domain of $\bR^n$. Then for any $r< r_x$, we have
\[
(1-\epsilon) \alpha(n)r^n \ls \cH^n(B(x_i,r)) \ls (1+\epsilon) \alpha(n)r^n.
\]
Now consider the set of balls
$$
\{B(x,\frac{1}{i})| x \in R_X, \frac{1}{r_x} <i< \infty\}
$$
By vitali covering theorem, we can choose disjoint balls $B(x_i,r_i)$ with $r_i\ls r_{x_i}$, such that $\cH^n(A \backslash \cup_i B(x_i,r_i))=0$.

For complete, n dimensional GCBA, by Theorem 1.2 and Corollary 11.2 of \cite{LyN2016}, for sufficiently small $\delta>0$, the set of  $(n,\delta)$-strained points has full measure. And for any $(n,\delta)$ strained point $x\in X$, there exists $r_x>0$, such that for any $r\ls r_x$,
\[
(1-\epsilon) \alpha(n)r^n \ls \cH^n(B(x,r)) \ls (1+\epsilon) \alpha(n)r^n.
\]
Consider the set of balls
$$
\{B(x,\frac{1}{i})| x \text{ is a }(n,\delta)-\text{strained point}, \frac{1}{r_x} <i< \infty\}
$$
By vitali covering theorem, we can choose disjoint balls $B(x_i,r_i)$ with $r_i \ls \frac{1}{j}$, $r_i\ls r_{x_i}$, such that $\cH^n(A \backslash \cup_i B(x_i,r_i))=0$.

For $C^{\alpha}$ Riemannian manifolds, for any $x\in X$, there exists $r_x$, a bi Lipschitz homeomorphism $\varphi$ from $B(x,r_x)$ onto a domain of $\bR^n$ such that for $y,z\in B(x,r_x)$,
\[
||yz|-|\varphi(y)\varphi(z)||\ls o(|yz|)|yz|,
\]
So we can repeat the above argument.

Deonte $B_{j,i}:=B(x_i,r_i)$, by Proposition \ref{prop:co1},
\[
\begin{array}{ll}
\int_{-\infty}^{\infty} \cH^{n-1} (f^{-1}(t)\cap A)dt-\int_{-\infty}^{\infty} \cH^{n-1} (f^{-1}(t) \cap \cup_i B_{j,i})dt
&\leqslant \int_{-\infty}^{\infty} \cH^{n-1} (f^{-1}(t)\cap (A\backslash \cup_i B_{j,i}))dt \\
&\leqslant C(n) \sup_{x\in A} Lip f(x) \cH^n (A\backslash \cup_i B_{j,i})\\
&=0.
\end{array}
\]
Then
\begin{equation}\label{in:11}
\begin{array}{ll}
\int_{-\infty}^{\infty} \cH^{n-1}_{\frac1j} (f^{-1}(t)\cap A)dt&=\int_{-\infty}^{\infty} \cH^{n-1}_{\frac1j} (f^{-1}(t) \cap \cup_i B_{j,i})dt\\
&=\int_{-\infty}^{\infty} \sum_i \chi_{f(B_{j,i})}(t)\cH^{n-1}_{\frac1j} (f^{-1}(t) \cap B_{j,i})dt\\
&\ls \int_{-\infty}^{\infty} \sum_i \chi_{f(B_{j,i})}(t) C_{\cH}(n-1)(diam B_{j,i})^{n-1}dt \\
&=\sum_i \int_{-\infty}^{\infty} \chi_{f(B_{j,i})}(t) dt C_{\cH}(n-1)(diam B_{j,i})^{n-1}\\
&\mathop{\ls}\limits^{(\ref{in:chi})} \sum_i Lip(f,B_{j,i})diam B_{j,i} C_{\cH}(n-1)(diam B_{j,i})^{n-1}\cdot \\
&\ls \sum_i  C_{\cH}(n-1) Lip(f,B_{j,i}) (2r_i)^n\\
&\mathop{\ls}\limits^{(\ref{in:reg}), (\ref{in:LipfB})}  2^{n+1}(1-\epsilon)C_{\cH}(n-1)(\alpha(n))^{-1} \inf_{x\in B_{j,i}}Lip f(x) \cH^n(B_{j,i})\\
&\ls C(n)\sum_i \int_{B_{j,i}} Lip f(x) dvol\\
&=C(n)\int_A Lip f(x) dvol.
\end{array}
\end{equation}
Let $\frac{1}{j}\to 0$, then we get (\ref{in:co1}).
\end{proof}

By a standard argument, we can get the following coarea inequality, see e.g. 3.4.3 of \cite{finepro}.
\begin{prop}
Let $f$ be a locally Lipschitz function and for $\cH^n$-a.e. $x\in X$, $Lip f(x)>0$, then for each $\cH^n$-measurable subset $A \subset X$,
\begin{equation}\label{in:coarea}
\int_{-\infty}^{\infty} \int_{A\cap f^{-1}(t)} \frac{1}{Lip f(x)} d\cH^{n-1}(x) dt \ls C(n) \cH^n(A).
\end{equation}
\end{prop}

Since for a locally Lipschitz, semiconcave function $f$, $|\nabla_x f|=Lip f(x)$, we have the following Corollary,
\begin{cor}\label{cor:co}
Let $f$ be a locally Lipschitz, semiconcave function and for $\cH^n$-a.e. $x\in X$, $|\nabla_x f|>0$, then for each $\cH^n$ measurable subset $A \subset X$, we have \begin{equation}\label{in:co}
\int_{-\infty}^{\infty} \int_{A\cap f^{-1}(t)} \frac{1}{|\nabla_x f|} d\cH^{n-1}(x) dt \ls C(n) \cH^n(A).
\end{equation}
\end{cor}

\section{Proof of the theorems}
In this section, a gradient curve will be denoted by $\alpha(t)$. $\Phi_f^t:X\to X$ is the f-gradient flow. $\Phi_f^t(x)=\alpha(t)$, where $\alpha(t)$ is a gradient curve with $\alpha(0)=x$.
Gradient curve $\alpha(t)$ reparametrized by arc length, will be denoted by $\bar{\alpha}(t)$, i.e. $\bar{\alpha}^+(t)=1$.
Denote $\tilde{\alpha}(t)$ the gradient curve reparametrized by $\tilde{\alpha}^+(t)=\frac{\nabla_{\tilde{\alpha}(t)} f}{|\nabla_{\tilde{\alpha}(t)}|^2}$.

Denote $\tilde{\beta}(t)$ the backward gradient curve reparametrized by $\tilde{\beta}^+(t)=\frac{\nabla_{\tilde{\beta}(t)} f}{|\nabla_{\tilde{\beta}(t)}|^2}$.

\begin{thm}
Let $X$ be one of: Complete, n dimensional SCBBL without boundary; Complete GCBA; $C^{\alpha}$ H\"older Riemannian manifold without boundary.
Let $f:X \to \bR$ be a concave function. Given a real number $a<\sup_{x\in X} f(x)$. For any points $p, q \in f^{-1}(a)$. Denote $\alpha_1(t)$ $(\alpha_2(t))$ the gradient curves with $\alpha_1(0)=p$, $\alpha_2(0)=q$, then the reparametrized gradient curves satisfy:
\begin{equation}\label{in:contraction}
|\tilde{\alpha}_1(t),\tilde{\alpha}_2(t)|\leqslant |pq|
\end{equation}
\end{thm}
\begin{proof}
Denote by $l(t)=|\tilde{\alpha}_1(t)\tilde{\alpha}_2(t)|$. First, we show that $l(t)$ is locally Lipschitz. In fact, for $0\leqslant t_1<t_2$, we have
\[
\begin{array}{ll}
||\tilde{\alpha}_1(t_2)\tilde{\alpha}_2(t_2)|-|\tilde{\alpha}_1(t_1)\tilde{\alpha}_2(t_1)||
&\leqslant |\tilde{\alpha}_1(t_1)\tilde{\alpha}_1(t_2)|+|\tilde{\alpha}_2(t_1)\tilde{\alpha}_2(t_2)|\\
&=\int_{t_1}^{t_2}\frac{1}{|\nabla_{\tilde{\alpha}_1(t)}f|}+\int_{t_1}^{t_2}\frac{1}{|\nabla_{\tilde{\alpha}_2(t)}f|}\\
&\leqslant(\frac{1}{|\nabla_{\tilde{\alpha}_1(t_2)}f|}+\frac{1}{|\nabla_{\tilde{\alpha}_2(t_2)}f|} )(t_2-t_1)
\end{array}
\]
Since $f\circ \tilde{\alpha}_i(t)-f(p)=\int_0^t \langle \nabla_{\tilde{\alpha}_i(s)} f, \frac{\nabla_{\tilde{\alpha}_i(s)} f}{|\nabla_{\tilde{\alpha}_i(s)}|^2} \rangle ds =\int_0^t 1=t$. We have $f\circ \tilde{\alpha}_1 (t)=f\circ\tilde{\alpha}_2(t)$. Let $\gamma$ be a geodesic connecting $\tilde{\alpha}_1(t)$ with $\tilde{\alpha}_2(t)$, let $\xi \in T_{\tilde{\alpha}_1(t)} $ (respectively, $\eta\in T_{\tilde{\alpha}_2(t)}$) be the original (respectively, the terminal) direction of $\gamma$.
\begin{equation}
0=f\circ \tilde{\alpha}_2(t)-f\circ \tilde{\alpha}_1(t)\leqslant d_{\tilde{\alpha}_1(t)}f(\xi) |\tilde{\alpha}_1(t)\tilde{\alpha}_2(t)| \leqslant \langle \nabla_{\tilde{\alpha}_1(t)} f, \xi \rangle |\tilde{\alpha}_1(t)\tilde{\alpha}_2(t)|.
\end{equation}
Then $\langle \nabla_{\tilde{\alpha}_1(t)} f, \xi \rangle \geqslant 0$.
\begin{equation}
0=f\circ \tilde{\alpha}_1(t)-f\circ \tilde{\alpha}_1(t)\leqslant d_{\tilde{\alpha}_2(t)}f(\eta) |\tilde{\alpha}_1(t)\tilde{\alpha}_2(t)| \leqslant \langle \nabla_{\tilde{\alpha}_2(t)} f, \eta \rangle |\tilde{\alpha}_1(t)\tilde{\alpha}_2(t)|.
\end{equation}
Then $\langle \nabla_{\tilde{\alpha}_2(t)} f, \eta \rangle \geqslant 0$. $l(t)$ is differentiable for almost every t, if $l(t)$ is differentiable at t, by the first variation inequality (\ref{in:first}),
\[
l^+(t) \leqslant -\langle \xi, \frac{\nabla_{\tilde{\alpha}_1(t)} f}{|\nabla_{\tilde{\alpha}_1(t)}|^2}\rangle-\langle \eta, \frac{\nabla_{\tilde{\alpha}_2(t)}f}{|\nabla_{\tilde{\alpha}_2(t)}f|^2}\rangle \leqslant 0.
\]
Then $l(t)$ is non-increasing.
\end{proof}

In the rest of this section, we assume that $X$ is one of: Complete, n dimensional SCBBL without boundary; Complete, n dimensional GCBA; $C^{\alpha}$ H\"older Riemannian manifold without boundary. Now we prove the three theorems in the introduction.

\textbf{Idea of proof}
Choose $p\in X$ and real number $r>0$ sufficiently small such that $B(p,r)$ is bi-Lipschitz to an open subset of $\bR^n$ and for any $q\in B(p,r)$, $|\nabla_q f|\neq 0$. Choose $t\in f(B(p,r))$, such that $S_0:=f^{-1}(t)\cap \overline{B(p,r)}$ is $\cH^{n-1}$-measurable. Without loss of generality, we can assume that $t=0$. We consider the reparametrized backward gradient curve $x\in S_0 \to \tilde{\beta}_x$. $\tilde{\beta}_x$ may not be unique. However, by a measurable selection argument, there exists a subset $S_0'\subset S_0$, for any $x\in S_0'$, we can choose a single backward gradient curve $\tilde{\beta}_x$ such that $A'=\{y|y\in \cup_{x\in S_0'} \tilde{\beta}_x([0,\infty))$ is a closed subset, hence $\cH^n$-measurable, then we can apply the coarea inequality:
\[
\begin{array}{ll}
\cH^n(X) &\gs \cH^n(A')\\
&\gs C(n) \int_{-\infty}^0 \int_{A'\cap f^{-1}(t)} \frac{1}{|\nabla_y f|}d\cH^{n-1} dt.
\end{array}
\]
For $x_1,x_2 \in S_0$, $\tilde{\beta}_{x_1}(t), \tilde{\beta}_{x_2}(t)\in A'\cap f^{-1}(t)$, and $|\tilde{\beta}_{x_1}(t)\tilde{\beta}_{x_2}(t)|\gs |x_1x_2|$. So the map $x \to \tilde{\beta}_{x}(t)$ is distance expanding. Then
\[
\begin{array}{ll}
\int_{-\infty}^0 \int_{A'\cap f^{-1}(t)} \frac{1}{|\nabla_y f|}d\cH^{n-1} dt
&\gs \int_{-\infty}^0 \int_{S_0'} \frac{1}{|\nabla_{\tilde{\beta}_x(t)} f|} d\cH^{n-1} dt\\
&=\int_{S_0'} \int_{-\infty}^0  \frac{1}{|\nabla_{\tilde{\beta}_x(t)} f|} dt d\cH^{n-1}.
\end{array}
\]
Note that $\int_0^{\infty}  \frac{1}{|\nabla_{\tilde{\beta}_x(t)} f|} dt $ is the length of the curve $\tilde{\beta}_x([0,\infty))$, which is $+\infty$. So $\cH^n(X)=+\infty$.

The following is the proof in detail, we will choose a closed subset $A'\subset X$ by a measurable selection argument.
\begin{proof}
Denote $S:=\{x\in X| f(x)=\sup_{y\in X} f(y)\}$, $S$ may be empty. Choose a point $p\notin S$, such that there exists a ball $B(p,r)$, and a bi-Lipschitz homeomorphism $\varphi: B(p,r)\to \varphi(B(p,r))\subset \bR^n$. Then $f\circ \varphi^{-1}: \varphi(B(p,r)) \to R$ is a locally Lipschitz function. By Theorem 3.2.15 of \cite{federer1969}, for $\cL^1$-a.e. t, $(f\circ \varphi^{-1})^{-1}(t) \cap \varphi(B(p,r))$ is countably $\cH^{n-1}$ rectifiable. Fix such a $t$ (W.L.O.G. we can assume $t=0$), such that $\cH^{n-1}((f\circ \varphi^{-1})^{-1}(0) \cap \varphi(B(p,r)))>0$. Since $\varphi: B(p,r) \to \bR^n$ is bi-Lipschitz and
$$
f^{-1}(0)\cap B(p,r)=\varphi^{-1}[\cH^{n-1}((f\circ \varphi^{-1})^{-1}(0) \cap \varphi(B(p,r))].
$$
We have $\cH^{n-1}(f^{-1}(0)\cap B(p,r))>0$.

Denote $S_0:=f^{-1}(0)\cap \overline{B(p,r)}$ and $v_0=\cH^{n-1}(f^{-1}(0)\cap B(p,r))$. Denote $\tilde{\Phi}_f^t:X \to X$ by $\tilde{\Phi}_f^t(x)=\tilde{\alpha}(t)$ with $\tilde{\alpha}(0)=x$.

Denote $S_t:=\{x\in X: \tilde{\Phi}_f^t (x)\in S_0\}$. Denote $A:=\cup_{t\in[0,\infty)} S_t$, i.e. $A$ consists of the points on the backward gradient curves with initial points in $S_0$. Denote $A_{t_0}=\cup_{t\in [0,t_0]} S_t$.

Denote $G_{t_0}$ the set of reparametrized backward gradient curves $\tilde{\beta}: [0,t_0]\to A_{t_0}$ with $\tilde{\beta}(0)\in S_0$. Then $\tilde{\beta}^{-}(t)=\frac{\nabla_{\tilde{\beta}(t)} f}{|\nabla_{\tilde{\beta}(t)} f|^2}$, $\tilde{\beta}(t_0)\in S_{t_0}$.

For $\tilde{\beta}_1, \tilde{\beta}_2 \in G_{t_0}$, consider the metric
 \[
 d_{t_0}(\tilde{\beta}_1, \tilde{\beta}_2):= \sup_{t\in [0,t_0]} |\tilde{\beta}_1, \tilde{\beta}_2|=|\tilde{\beta}_1(t_0), \tilde{\beta}_2(t_0)|.
 \]
\begin{lem}
$(G_{t_0},d_{t_0})$ is a complete and separable metric space.
\end{lem}

\begin{proof}[proof of the above lemma]
First we show that $(G_{t_0},d_{t_0})$ is complete. Suppose that $\tilde{\beta}_i$ is a Cauchy sequence. Then there exists a subsequence $\tilde{\beta}_{i_j}(t)$ converging to some curve $\tilde{\beta}_{\infty}(t):[0,t_0]\to A_{t_0}$. We claim that $\tilde{\beta}_{\infty} (t)\in G_{t_0}$.
Denote $\tilde{\alpha}_i(t)=\tilde{\beta}_i(t_0-t)$, $\tilde{\alpha}_{\infty}(t)=\tilde{\beta}_{\infty}(t_0-t)$. Then $\tilde{\alpha}_i(t)$ converge to $\tilde{\alpha}_{\infty}(t)$. Denote $\tilde{\alpha}(t)$ the gradient curve parametrized by $\tilde{\alpha}^+(t)=\frac{\nabla_{\tilde{\alpha}(t)}f}{|\nabla_{\tilde{\alpha}(t)}f|^2}$ with $\tilde{\alpha}(0)=\tilde{\alpha}_{\infty}(0)\in S_{t_0}$. By inequality (\ref{in:contraction}), $|\tilde{\alpha}_i(t) \tilde{\alpha}(t)|\ls |\tilde{\alpha}_i(0) \tilde{\alpha}(0)|$, then $\tilde{\alpha}_i(t)$ converge to $\tilde{\alpha}(t)$. So $\tilde{\alpha}_{\infty}(t)=\tilde{\alpha}(t)$, then $\tilde{\beta}_{\infty}(t)\in G_{t_0}$.

Next, we show that $(G_{t_0}, d_{t_0})$ is separable. Since $S_{t_0}$ is compact, we choose countable dense subset $\{x_i\}_{i=1}^{\infty} \subset S_{t_0}$. Consider the gradient curve $\tilde{\alpha}_i:[0,t_0] \to A_{t_0}$ parametrized by $\tilde{\alpha}^+(t)=\frac{\nabla_{\tilde{\alpha}(t)}f}{|\nabla_{\tilde{\alpha}(t)}f|^2}$ with $\tilde{\alpha}(0)=x_i$. By inequality (\ref{in:contraction}),
\[
G'_{t_0}:=\{\tilde{\beta}_i(t):=\tilde{\alpha}_i(t_0-t)\}
\]
is a countable, dense subset of $G_{t_0}$.
\end{proof}

Consider the multifunction $F:S_0\to A_{t_0}$, for $x \in S_0$,
\[
F(x)=\{\tilde{\beta} \in G_{t_0}, \tilde{\beta}(0)=x\}.
\]
Then $F(x)$ is a closed subset of $G_{t_0}$. By a selection theorem (see e.g. Theorem 5.21 of book \cite{Sriva1998}), there exists a Borel measurable map $g: S_0 \to G_{t_0}$. By Lusin's theorem, for any $\epsilon>0$, there exists a closed subset $S^1_0\subset S_0$ with $\cH^{n-1}(S_0\backslash S_0^1)<\epsilon$ such that g restricted to $S_0^1$ is continuous. Denote $S_{t_0}^1:=g(S_0^1)\cap S_{t_0}$, then it's a compact subset with $\cH^{n-1}(S_{t_0}^1)>\cH^{n-1}(S_0^1)$.

For $t_1>t_0$, denote
\[
A_{t_0,t_1}:=\{x\in X| \tilde{\Phi}_f^t(x)\in S_0 \text{ for some } t \text{ with } t_0\ls t\ls t_1\},
\]
i,e. $A_{t_0,t_1}=\cup_{t\in [t_0,t_1]} S_t$. Denote
\[
G_{t_0,t_1}:=\{\tilde{\beta}:[t_0,t_1]\to A_{t_0,t_1}| \tilde{\beta}(t_0)\in S_{t_0} \}
\]
By repeating the above arguments, there exists $S_{t_0}^2\subset S_{t_0}^1$ with $\cH^{n-1}(S_{t_0}^1 \backslash S_{t_0}^2)<\frac{\epsilon}{2}$, and a continous map $g_1: S_{t_0}^2 \to G_{t_0,t_1}$. Denote $S_{t_1}^1:=g_1(S_{t_0}^2)\cap S_{t_1}$, then $\cH^{n-1}(S_{t_1}^1)\gs \cH^{n-1}(S_{t_0}^2)$.

Choose an increasing sequence $t_i\to \infty$, by repeating the above arguments, there exist closed subsets $S_{t_i}^2\subset S_{t_i}^1 \subset S_{t_i}$ with $\cH^{n-1}(S_{t_i}^1\backslash S_{t_i}^2)<\frac{\epsilon}{2^i}$ and continuous maps $g_i: S^2_{t_{i-1}}\to A_{t_{i-1}, t_i}$, $S_{t_i}^1=g_i(S^2_{t_{i-1}})\cap S_{t_i}$. $\cH^{n-1}(S_{t_i}^1)\gs \cH^{n-1}(S_{t_{i-1}}^2)$.

Denote $S_0'=S_0^1\cap_{i=0}^{\infty} \tilde{\Phi}_f^{t_i}(S^2_{t_i})$, then $S_0'$ is closed and
\[
S_0':=\{x\in S_0|\text{ there exists }\tilde{\beta} \text{ with } \tilde{\beta}(0)=x, \tilde{\beta}(t_i)\in S^2_{t_i}\text{ for all }i \}.
\]
Then $S'_0\supset (S_0\backslash \cup_i \tilde{\Phi}_f^{t_i} (S_{t_i}^1\backslash S_{t_i}^2))$.
\[
\begin{array}{ll}
\cH^{n-1}(S_0')&\gs \cH^{n-1}(S_0^1)-\sum_i \cH^{n-1}(S_{t_i}^1\backslash S_{t_i}^2)\\
&\gs \cH^{n-1}(S_0)-100 \epsilon.
\end{array}
\]
Denote
\[
A'=\{x\in X: x\in \tilde{\beta}([0,\infty)) |\tilde{\beta}(0)\in S_0' \}
\]
Then
\[
A'=\{x\in \cap_i \tilde{\Phi}_f^{t}(S_{t_i}^2), t\in [0,t_i]\},
\]
so $A'$ is a closed subset of $X$, hence $\cH^{n}$-measurable. $\tilde{\beta}_x(t)$ is the backward gradient curve with $\tilde{\beta}(0)=x$. By theorem \ref{thm:out}, $\tilde{\beta}(t)$ can be defined on $t\in [0,\infty)$ and the length $Length(\tilde{\beta}_x[0,\infty))=\infty$. Then by Corollary \ref{cor:co}, for a.e. t, $A'\cap f^{-1}(t)$ is $\cH^{n-1}$-measurable and
\[
\begin{array}{ll}
\cH^n(X)&>\cH^n(A')\\
&\gs \int_{-\infty}^0 \int_{A'\cap f^{-1}(t)} \frac{1}{|\nabla_x f|} d\cH^{n-1} dt\\
&\gs \int_{-\infty}^0 \int_{S_0'} \frac{1}{|\nabla_{\tilde{\beta}_x(t)} f|} d\cH^{n-1} dt\\
&=\int_{S_0'} \int_{-\infty}^0  \frac{1}{|\nabla_{\tilde{\beta}_x(t)} f|} dt d\cH^{n-1}\\
&=\int_{S_0'} Length(\tilde{\beta}_x[0,\infty)) d\cH^{n-1}\\
&=\infty.
\end{array}
\]
\end{proof}

\end{document}